\documentclass{amsart}
\usepackage[utf8]{inputenc}
\usepackage{graphicx}

\usepackage{nir}
\title{Prime Sums}
\author{Anupam Datta}

\author{Nir Elber}
\email{nire@berkeley.edu}

\author{Raymond Feng}

\author{David Lowry-Duda}
\email{davidlowryduda@brown.edu}

\author{Henry Xie}

\address{ICERM, 121 South Main Street, Box E, 11th Floor, Providence, RI 02903}
\thanks{DLD was supported by the Simons Collaboration in Arithmetic
Geometry, Number Theory, and Computation via the Simons Foundation grant
546235.}
\date{\today}

\begin{document}

\maketitle

\begin{abstract}
    We study the properties of certain graphs involving the sums of primes. Their structure largely turns out to relate to the distribution of prime gaps and can be roughly seen in Cram\'er's model as well. We also discuss generalizations to the Gaussian integers.
\end{abstract}

\tableofcontents

\newpage

\section{Introduction}

The main motivation is why the following figure has so much structure.
\begin{figure}[ht]
    \centering
    \includegraphics[width=0.7\textwidth]{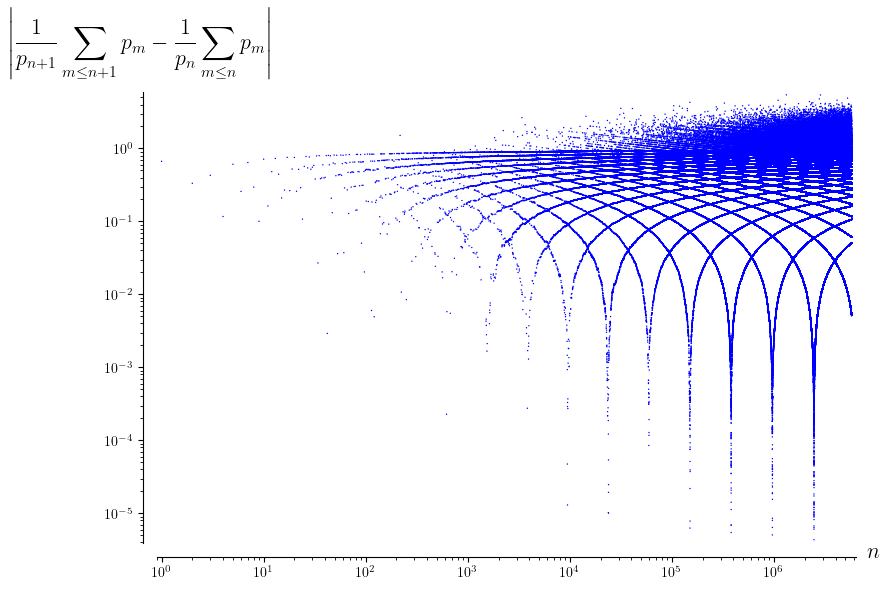}
    \caption{Log-log plot of absolute consecutive differences of $a_n:=\frac1{p_n}\sum_{m\le n}p_m$}
    \label{fig:motivation}
\end{figure}

\noindent There are a number of features of the above figure which stand out here: the relatively constant horizontal line at the top which gives way to a more chaotic ``cloud'' of points, the ``troughs'' which appear to occur at somewhat regular intervals, the various other curves which seem to repeat themselves, and so on. Each of these aspects of the graph will be studied in their own right.

With this in mind, the layout of this paper is as follows. In \autoref{sec:graph}, we study various aspects of \autoref{fig:motivation}. The highlights are that most of the repeating structure comes from prime gaps over various sizes, where the key lemma is as follows.
\begin{restatable}[Key lemma]{prop}{keylemma} \label{prop:goestohalf}
    We have that
    \[\lim_{n\to\infty}\frac{\log n}{p_n-p_{n+1}}\left(\frac1{p_{n+1}}\sum_{m\le n}p_m-\frac1{p_n}\sum_{m\le n}p_m\right)=\frac12.\]
\end{restatable}
\noindent Relating \autoref{prop:goestohalf} with \autoref{fig:motivation}, lets us more finely discuss its structure. For example, we are able to prove the following.
\begin{restatable}{prop}{troughs} \label{prop:troughs}
    As $n\to\infty,$ the location of the troughs (provided the troughs exist) in \autoref{fig:goodloglog} occur at $n\approx e^k$ where $k$ varies over positive integers.
\end{restatable}
\begin{remark}
    It should be noted that this is simply stating where the troughs occur only {when they exist}; it is unknown whether there is a trough for every $k$ (this requires guaranteeing the existence of a specific prime gap value around a certain region of primes), or whether any of the troughs extend infinitely to the right (de Polignac's conjecture \parencite{polignac}, if true, would imply this statement because every even integer would appear infinitely often as a prime gap).
\end{remark}

Then in \autoref{sec:random}, we quickly compare \autoref{fig:motivation} with a variant of Cram\'er's model and find that, even though what we can prove about the error term is stronger with Cram\'er's model (in \autoref{prop:cramer}), in practice, the primes appear more structured.

Lastly, in \autoref{sec:gaussian}, we extend some of these ideas to the Gaussian integers. Even though studying even the main term of prime sums over $\ZZ[i]$ appears quite difficult when attempting to keep track of the angle, we are able to achieve something. Namely, we are able to show the following.
\begin{restatable}{thm}{gaussianthm} \label{thm:gaussian}
    Fix $n$ a nonzero integer. Then, for any real number $x,$
    \[\sum_{\op N(\pi)\le x}\frac{\pi^n}{\op N(\pi)^{n/2}}=O\left(\frac{x}{(\log x)^A}\right),\]
    where $A$ is an arbitrarily large (but fixed) constant. Here, the sum is taken over Gaussian primes $\pi\in\ZZ[i]$ with $\op N(\pi)\le x$; the sum is a real number.
\end{restatable} 
\noindent The ideas which go into the proof of \autoref{thm:gaussian} motivate a random model for the Gaussian primes as well, which we discuss at the end of the section.

\section{Graphing Large Values} \label{sec:graph}

As promised, the main cause of the structure of \autoref{fig:motivation} is the following asymptotic.
\keylemma*
\begin{proof}
    This is a culmination of some known results. After factoring out the sum and combining the fractions, we are computing
    \[\lim_{n\to\infty}\frac{\log n}{p_n-p_{n+1}}\cdot\frac{p_n-p_{n+1}}{p_{n+1}p_n}\sum_{m\le n}p_m.\]
    Here, the (negative) prime gap $p_n-p_{n+1}$ will cancel out. It is also known that $\sum_{m\le n}p_m\sim\frac12n^2\log n$ and that $p_n\sim p_{n+1}\sim n\log n$; plugging these all in gives
    \[\lim_{n\to\infty}\frac{\log n}1\cdot\frac1{(n\log n)(n\log n)}\cdot\frac12n^2\log n=\frac12.\]
    This finishes the proof.
\end{proof}
\begin{remark}
    The convergence of this function is actually very slow (see \autoref{fig:goestohalf}), but the function itself appears quite smooth. We will discuss this below.
\end{remark}
To analyze the convergence of the above quantity, we prove the following proposition:
\begin{prop} \label{prop:error}
    We actually have \[\frac{\log n}{p_np_{n+1}}\sum_{m\le n}p_m=\frac12-\frac{\log\log n}{2\log n}+\frac1{4\log n}+o\left(\frac1{\log n}\right).\]
\end{prop}
\begin{proof}
    We use the bounds of \[p_n=n(\log n+\log\log n-1+o(1))\] and \[\sum_{m\leq n}p_m=\frac{n^2}2\left(\log n+\log\log n-\frac32+o(1)\right),\]which can be found in \parencite{sinha2015asymptotic}. Note that
    \begin{align*}
        \log(n+1)+\log\log(n+1)-1+o(1)&=\log n+\log\left(1+\frac1n\right)+\log\log n+\log\left(\frac{\log(n+1)}{\log n}\right)-1+o(1)\\
        &=\log n+\log\log n-1+o(1).
    \end{align*}
    Also, we have
    \begin{align*}
        \frac{\log n}{\log n+\log\log n-1+o(1)}&=1-\frac{\log\log n-1+o(1)}{\log n+\log\log n-1+o(1)}\\
        &=1-\frac{\log\log n-1}{\log n}+(\log\log n-1)\cdot\left[\frac1{\log n}-\frac1{\log n+\log\log n-1}\right]\\
        & \qquad +o\left(\frac1{\log n}\right)\\
        &=1-\frac{\log\log n-1}{\log n}+o\left(\frac1{\log n}\right).
    \end{align*}
    Then we have that
    \begin{align*}
        \frac{\log n}{p_np_{n+1}}\sum_{m\le n}p_m&=\frac{\log n}{(n+1)(\log n+\log\log n-1+o(1))}\cdot\frac{\frac{n^2}2\left(\log n+\log\log n-\frac32+o(1)\right)}{n(\log n+\log\log n-1+o(1))}\\
        &=\frac{\log n}{(n+1)(\log n+\log\log n-1+o(1))}\cdot\frac n2 \\
        & \qquad\cdot\left[1-\frac1{2(\log n+\log\log n-1+o(1))}+o\left(\frac1{\log n}\right)\right]\\
        &=\left(1-\frac{\log\log n-1}{\log n}+o\left(\frac1{\log n}\right)\right)\left(\frac12-\frac1{2n+2}\right)\cdot\left[1-\frac1{2\log n}+o\left(\frac1{\log n}\right)\right]\\
        &=\frac12-\frac1{4\log n}-\frac{\log\log n-1}{2\log n}+o\left(\frac1{\log n}\right)\\
        &=\frac12-\frac{\log\log n}{2\log n}+\frac1{4\log n}+o\left(\frac1{\log n}\right),
    \end{align*}
    as claimed.
\end{proof}
\begin{figure}[!htbp]
    \centering
    \includegraphics[width=0.7\textwidth]{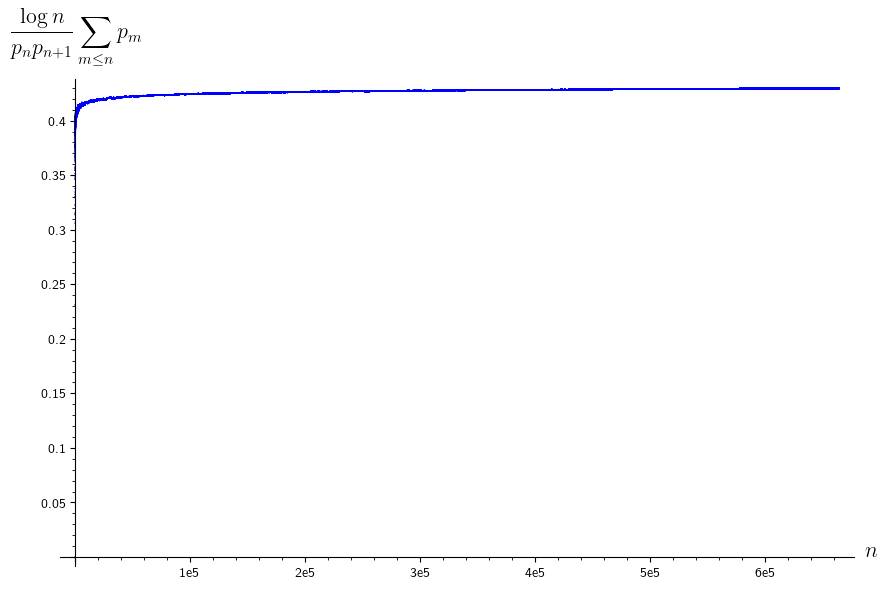}
    \caption{Convergence of \autoref{prop:goestohalf}}
    \label{fig:goestohalf}
\end{figure}
\begin{figure}[!htbp]
    \centering
    \includegraphics[width=0.7\textwidth]{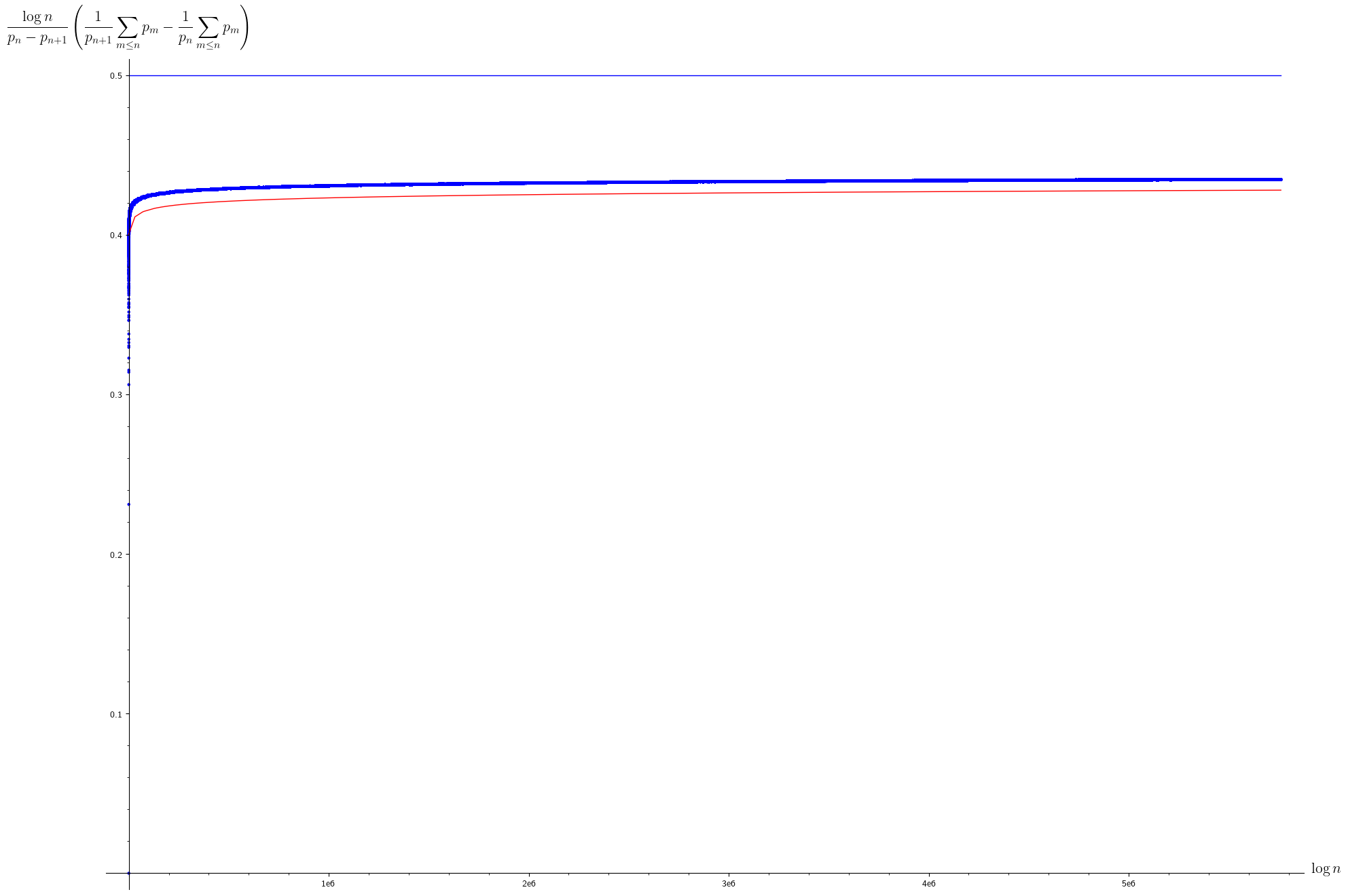}
    \caption{Convergence of \autoref{prop:goestohalf} with asymptotic estimate in red}
    \label{fig:slowconverge}
\end{figure}

We now begin to unravel \autoref{prop:goestohalf} to talk about the structure of \autoref{fig:motivation}. To begin, we see that
\[\log n\left(\frac1{p_{n+1}}\sum_{m\le n}p_m-\frac1{p_n}\sum_{m\le n}p_m\right)\sim\frac{p_n-p_{n+1}}2,\]
so graphing $y=\log n\left(\frac1{p_{n+1}}\sum_{m\le n}p_m-\frac1{p_n}\sum_{m\le n}p_m\right)$ with respect to $x=n$ will look like descending horizontal lines, one corresponding to each prime gap. This won't change if we graph with respect to $x=\log n$; compare \autoref{fig:goestoprimegap}.
\begin{figure}[!htbp]
    \centering
    \includegraphics[width=0.7\textwidth]{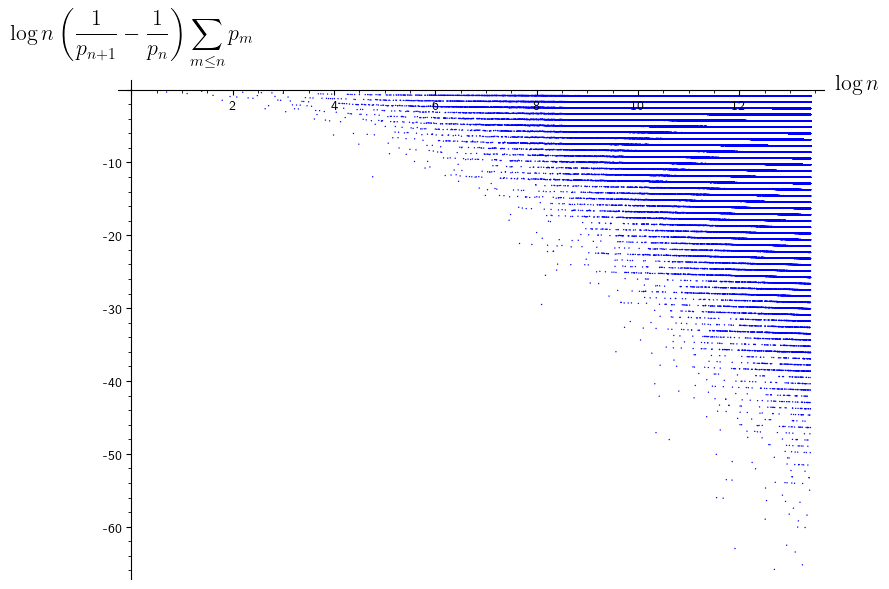}
    \caption{Prime gaps}
    \label{fig:goestoprimegap}
\end{figure}

Continuing, we next take out the $\log n$ factor. Because $y=\log n\left(\frac1{p_{n+1}}\sum_{m\le n}p_m-\frac1{p_n}\sum_{m\le n}p_m\right)$ as a function of $x=\log n$ looks like horizontal lines, graphing $y=\frac1{p_{n+1}}\sum_{m\le n}p_m-\frac1{p_n}\sum_{m\le n}p_m$ with respect to $x=\log n$ will look like
\[xy\sim\frac{p_n-p_{n+1}}2.\]
Equivalently, this looks like
\[y\sim\frac{p_n-p_{n+1}}2\cdot\frac1x.\]
Namely, these are descending hyperbolas, one hyperbola for each prime gap. See \autoref{fig:hyperbolas}.
\begin{figure}[!htbp]
    \centering
    \includegraphics[width=0.7\textwidth]{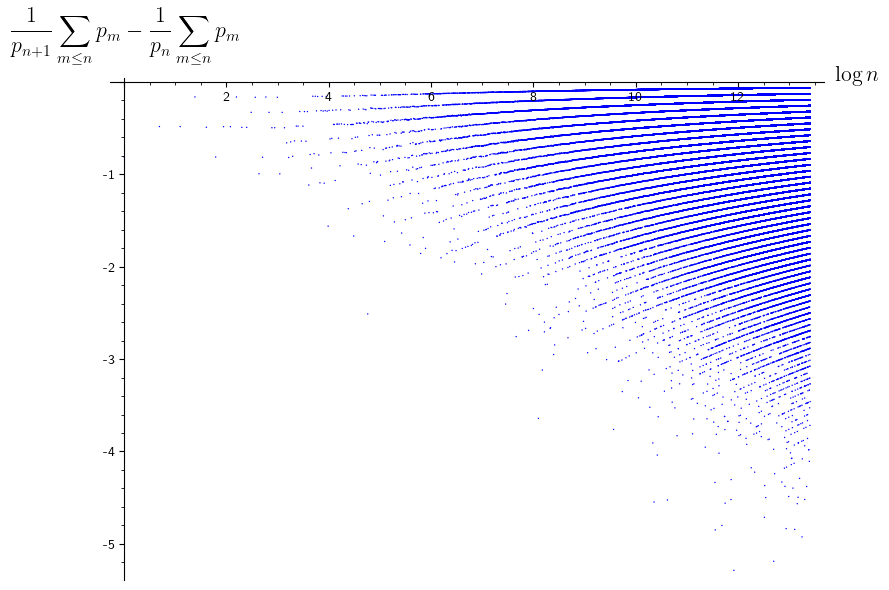}
    \caption{Hyperbolas}
    \label{fig:hyperbolas}
\end{figure}

We now correct the index of $\frac1{p_{n+1}}\sum_{m\le n}p_m$ by simply adding $1=p_{n+1}/p_{n+1}$ to $y.$ This means we want to graph $y=\frac1{p_{n+1}}\sum_{m\le n+1}p_m-\frac1{p_n}\sum_{m\le n}p_m$ as a function of $x=\log n,$ which will look like
\[y=1+\frac{p_n-p_{n+1}}2\cdot\frac1x.\]
This shifts all hyperbolas up by $1,$ giving \autoref{fig:hyperbolasplusone}.
\begin{figure}[!htbp]
    \centering
    \includegraphics[width=0.7\textwidth]{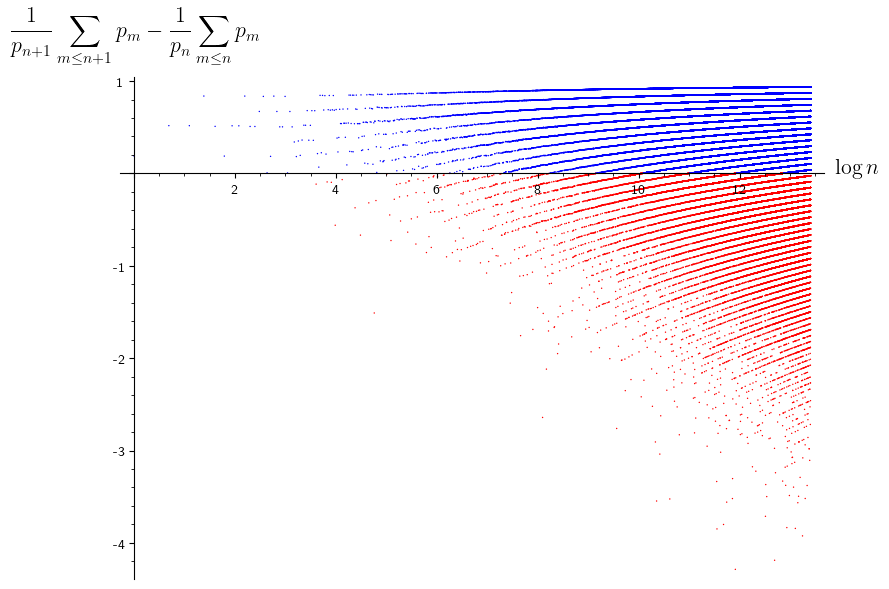}
    \caption{Hyperbolas shifted up by $1,$ with negatives colored red}
    \label{fig:hyperbolasplusone}
\end{figure}

Now, we see that a log-log plot of $\frac1{p_{n+1}}\sum_{m\le n+1}p_m-\frac1{p_n}\sum_{m\le n}p_m$ roughly corresponds to to graphing $y=\log\left|\frac1{p_{n+1}}\sum_{m\le n+1}p_m-\frac1{p_n}\sum_{m\le n}p_m\right|$ with respect to $x=\log n.$ This will look like
\[y=\log\left|1+\frac{p_n-p_{n+1}}2\cdot\frac1x\right|.\]
So we get a ``curtain'' of $\log\left|1-\frac kx\right|$ for each positive integer $k.$ (How clearly this current appears corresponds to how densely a particular prime gap appears.) This is actually what \autoref{fig:motivation}. \autoref{fig:goodloglog} is \autoref{fig:motivation} with the dots colored by sign.
\begin{figure}[!htbp]
    \centering
    \includegraphics[width=0.7\textwidth]{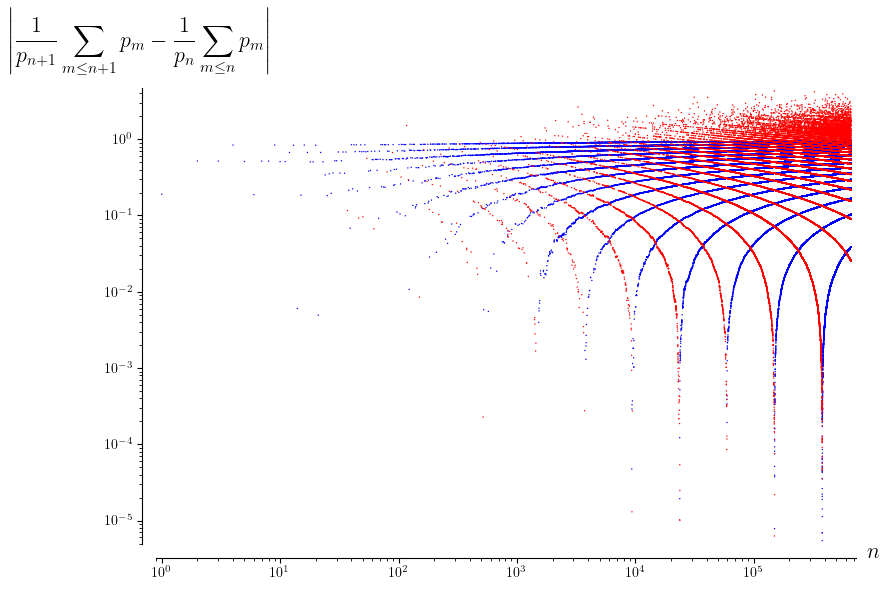}
    \caption{Log-log plot of $\left|\frac1{p_{n+1}}\sum_{m\le n+1}p_m-\frac1{p_n}\sum_{m\le n}p_m\right|,$ with negatives red.}
    \label{fig:goodloglog}
\end{figure}

At this point, we can begin to tease out structure from the graph of \autoref{fig:goodloglog}. For example, all of the positive differences (in blue), are converging to $10^0=1$ because the log-log plot should look like $\log\left|1-\frac kx\right|$ for various positive integers $k,$ and as $x\to\infty$ in such a plot, the value of this goes to $0,$ which is $10^0=1.$

More interestingly, we can predict where the ``troughs'' should be as $n$ grows large. This should occur, roughly speaking, when the graph of
\[y=\log\left|1-\frac{p_{n+1}-p_n}2\cdot\frac1x\right|\]
approaches $-\infty.$ Namely, fixing a particular prime gap $g:=p_{n+1}-p_n,$ we want $1-\frac{g/2}x$ to be roughly $0,$ which corresponds to $\log n=x\approx g/2.$ Thus, we have the following proposition.
\troughs*
\begin{proof}
    This follows from the above discussion.
\end{proof}

We can also discuss the red points above \autoref{fig:goodloglog}.
\begin{prop}
    As $n\to\infty$, the red cluster of points above the y-axis in \autoref{fig:goodloglog} continues to grow without bound.
\end{prop}
\begin{proof}
    Notice that \[\left|\frac1{p_{n+1}}\sum_{m\le n+1}p_m-\frac1{p_n}\sum_{m\le n}p_m\right|=1+\frac{\sum_{m\leq n}p_m}{p_np_{n+1}}\cdot(p_{n+1}-p_n)\sim \frac{p_{n+1}-p_n}2\cdot\frac1{\log n},\] but it is a known result \parencite{zbMATH02555765} that there are arbitrarily large primes $p$ such that \[p_{n+1}-p_n>c\log n\] for any constant $c$. In other words, \[\limsup_{n\to\infty}\frac{p_{n+1}-p_n}{\log n}=\infty.\] This implies that the difference attains arbitrarily large values, so the red cluster of points in \autoref{fig:goodloglog} has points of arbitrarily large $y$-value.
    
    On a somewhat related note, this also means that there are arbitrarily low/negative values in \autoref{fig:hyperbolasplusone}.
\end{proof}

There is also a related claim we can make about the largest values in \autoref{fig:hyperbolasplusone}, which correspond to the smallest prime gaps. The largest values are always less than one, as \[\frac1{p_{n+1}}\sum_{m\leq n+1}p_m-\frac1{p_n}\sum_{m\leq n}p_m=1-\frac{p_{n+1}-p_n}{p_{n+1}p_n}\cdot\sum_{m\leq n}p_n<1.\] On the other hand, we can show that there are values in \autoref{fig:hyperbolasplusone} that are arbitrarily close to 1, since it is known \parencite{goldston2007primes} that \[\liminf_{n\to\infty}\frac{p_{n+1}-p_n}{\log n}=0.\] In fact, much more is known; it is known now that $\liminf_{n\to\infty}|p_{n+1}-p_n|$ is finite \parencite{zbMATH06302171}.

\section{Comparing with Cram\'er's Model} \label{sec:random}

In this section we study whether the structure in the graphs of the prime sums from earlier can be attributed to the relatively random behavior of the prime numbers, or if the structure is special to the prime numbers.

To this end, our investigation focuses on the following modification of Cram\'er's random model: let $p_1=3$, and let $i=1$. For each odd number $k\geq 5$, starting at $k=5$ and incrementing by two each step, we do both of the following with probability $\frac2{\log k}$ (the factor of two comes from the fact that we are only choosing odd numbers, as we want to preserve the fact that the difference between consecutive terms is even):
\begin{itemize}
    \item Set $p_{i+1}=k$, and
    \item increment $i$ by 1.
\end{itemize}
In this manner, we generate a sequence of odd numbers which have the same asymptotic distribution as the prime numbers (this can be proven by the prime number theorem because we know that $\pi(n)\sim\frac n{\log n}$); in other words, for any interval of length $N$, the expected number of $p_i$ chosen by the procedure described above grows asymptotically at the same rate as the number of primes in the interval, as $N$ grows arbitrarily large.

As an example to compare, \autoref{fig:randomgoodloglog} is the corresponding version of \autoref{fig:goodloglog} for the random model.
\begin{figure}
    \centering
    \includegraphics[width=0.7\textwidth]{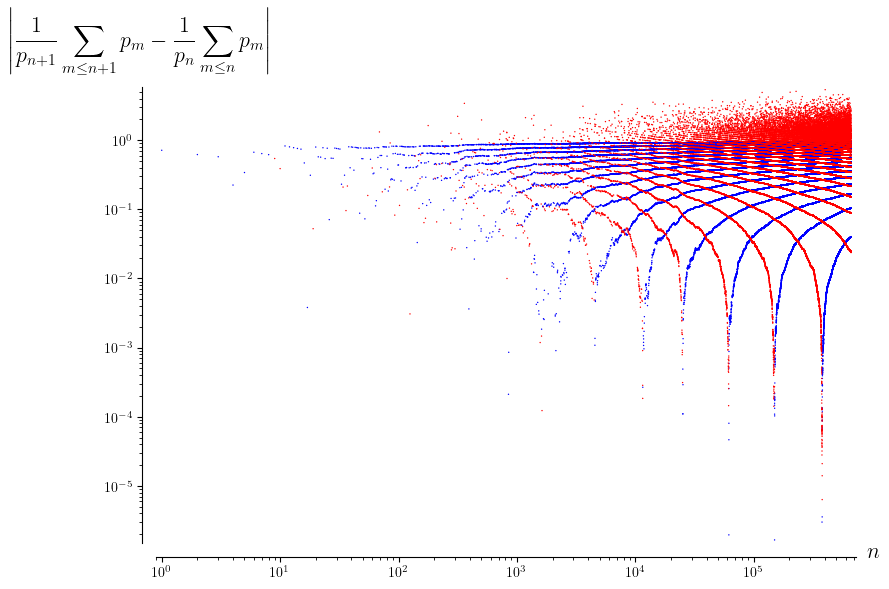}
    \caption{Randomized version of \autoref{fig:goodloglog}}
    \label{fig:randomgoodloglog}
\end{figure}
What is remarkable here is that \autoref{fig:randomgoodloglog} is somehow less stable than the case for primes; this is not a fluke. For example, the error terms which we can theoretically prove from Cram\'er's model in \autoref{thm:cramerpsums} also turn out to have incorrect constant factors in comparison to the actual primes.

The main point of this section is to establish the following result, in analogy with \autoref{prop:error}. We refer the reader to \autoref{sec:cramer} for the proof of the relatively known computation \autoref{thm:cramerpsums}.
\begin{prop} \label{prop:cramer}
    Fix $\varepsilon>0.$ Over all sequences $\mathcal P\subseteq\ZZ$ with probability weighted by Cram\'er's model, we order $\mathcal P=\{p_1,p_2,\ldots\}$ with $p_1<p_2<\cdots.$ Then we have that
    \[\frac{\log n}{p_np_{n+1}}\sum_{m\le n}p_m=\frac{\log n}{\op{Li}^{-1}(n)^2}\int_2^{\op{Li}^{-1}(n)}\frac t{\log t}\,dt+O_\varepsilon\left(x^{-1/2+\varepsilon}\right)\]
    holds almost surely. Here, $\op{Li}$ refers to the logarithmic integral.
\end{prop}
\begin{proof}
    This proof is mostly done by force, converting each term in the sum to its asymptotic according to \autoref{thm:cramerpsums}. As such, we will split this proof into many parts.
    \begin{lem}
        Fix everything as above. Then $p_n=\op{Li}^{-1}(n)+O_\varepsilon\left(n^{1/2+\varepsilon}\right)$ almost surely. In particular, $p_n=O_\varepsilon\left(n^{1+\varepsilon}\right).$
    \end{lem}
    \begin{proof}
        Most of the work here is done by \autoref{thm:cramerpsums}, which after plugging in $x=p_n$ yields
        \[\op{Li}(p_n)=\int_2^{p_n}\frac1{\log t}\,dt=n+O_\varepsilon\left(p_n^{1/2+\varepsilon}\right) \label{eq:pn} \tag{3.1}\]
        with probability $1-p_n^{-2}.$ We would like to take $\op{Li}^{-1}$ everywhere. This function is strictly increasing and, in fact, for real numbers $e<x<y,$ we have
        \[\op{Li}(x)-\op{Li}(y)=\int_y^x\frac1{\log t}\,dt\le x-y.\]
        In particular, we see that $|\op{Li}(x)-\op{Li}(y)|<|x-y|$ for any real numbers $x,y>e.$ Thus, we may use \autoref{eq:pn} to write
        \[\left|p_n-\op{Li}^{-1}(n)\right|<|\op{Li}(p_n)-n|=O_\varepsilon\left(p_n^{1/2+\varepsilon}\right).\]
        So we see that $p_n=\op{Li}^{-1}(n)+O_\varepsilon\left(p_n^{1/2+\varepsilon}\right).$
        
        Now, $n\sim\op{Li}(p_n)$ lets us get $p_n\sim n\log n$ in the same way that the Prime number theorem implies $p_n\sim n\log n,$ so we get $p_n=O\left(n^{1+\varepsilon}\right),$ implying $p_n=\op{Li}^{-1}(n)+O_\varepsilon\left(p_n^{1/2+\varepsilon}\right),$ where perhaps we have to switch out our constant factors.
    \end{proof}
    The above lemma, among other things, gives us relatively easy access to the fraction $\frac1{p_np_{n+1}}.$
    \begin{lem}
        Fix everything as above. We have that $\frac1{p_{n+1}p_n}=\frac1{\op{Li}^{-1}(n)^2}+O_\varepsilon\left(n^{-5/2+\varepsilon}\right)$ almost surely.
    \end{lem}
    \begin{proof}
        We begin by converting $1/(p_{n+1}p_n)$ into $1/p_n^2.$ We see
        \[\frac1{p_{n+1}p_n}=\frac1{p_n^2}-\frac1{p_n}\left(\frac1{p_n}-\frac1{p_{n+1}}\right)=\frac1{p_n^2}-\frac{p_{n+1}-p_n}{p_n^2p_{n+1}}.\]
        It is known \parencite{Cramr1936} that Cram\'er's model gives an estimate of $p_{n+1}-p_n=O\left((\log p_n)^2\right),$ so the error term here comes out to $O\left(p_n^{-3}(\log p_n)^2\right)=O_\varepsilon\left(n^{-3+\varepsilon}\right)$ by readjusting $\varepsilon$'s constant factors as necessary.
        
        We now convert $1/p_n^2$ to $1/\op{Li}^{-1}(n)^2.$ For this, we write
        \[\left|\frac1{p_n^2}-\frac1{\op{Li}^{-1}(n)^2}\right|=\left|\frac1{p_n}+\frac1{\op{Li}^{-1}(n)}\right|\cdot\left|\frac1{p_n}-\frac1{\op{Li}^{-1}(n)}\right|.\]
        The first term is the sum of two fractions, both of whose denominators are bigger than $n$ with probability $1-n^{-2},$ so we may bounds this as $O(1/n).$ We can estimate the second term as
        \[\frac{\left|\op{Li}^{-1}(n)-p_n\right|}{p_n\op{Li}^{-1}(n)}.\]
        The numerator is $O_\varepsilon\left(n^{1/2+\varepsilon}\right)$ while the denominator is $O\left(1/n^2\right),$ so we accumulate $O_\varepsilon\left(n^{-3/2+\varepsilon}\right)$ here. Combining the two terms, we get $O_\varepsilon\left(n^{-1}\cdot n^{-3/2+\varepsilon}\right),$ which is what we wanted.
    \end{proof}
    Lastly, we turn the sum into an integral.
    \begin{lem}
        Fix everything as above. We have that
        \[\sum_{m\le n}p_m=\int_2^{\op{Li}^{-1}(n)}\frac t{\log t}\,dt+O_\varepsilon\left(x^{3/2+\varepsilon}\right)\]
        almost surely.
    \end{lem}
    \begin{proof}
        Plugging in $p_n$ into \autoref{thm:cramerpsums}, we get that
        \[\sum_{m\le n}p_m=\int_2^{p_n}\frac{t}{\log t}\,dt+O_\varepsilon\left(p_n^{3/2+\varepsilon}\right)\]
        with probability $1-p_n^{-2}>1-n^{-2}.$ Plugging in our asymptotic $p_n=\op{Li}^{-1}(n)+O_\varepsilon\left(n^{1/2+\varepsilon}\right)$ (which holds in the same situation by partial summation), we get
        \[\int_2^{p_n}\frac t{\log t}\,dt=\int_2^{\op{Li}^{-1}(n)}\frac t{\log t}\,dt+O_\varepsilon\left(n^{1/2+\varepsilon}\cdot\frac{p_n}{\log p_n}\right).\]
        Now, the error term here will collapse into $O_\varepsilon\left(n^{3/2+\varepsilon}\right)$ once we use $p_n=O\left(n^{1+\varepsilon}\right)$ and readjust the constant factors. Combining our error terms finishes.
    \end{proof}
    We are now ready to prove the proposition. To begin, we write
    \[\sum_{m\le n}p_m=\int_2^{\op{Li}^{-1}(n)}\frac t{\log t}\,dt+O_\varepsilon\left(x^{3/2+\varepsilon}\right)\]
    and multiply both sides by $1/(p_np_{n+1})=1/\op{Li}^{-1}(n)^2+O_\varepsilon\left(n^{-5/2+\varepsilon}\right)$ to get
    \[\frac1{p_np_{n+1}}\sum_{m\le n}p_m=\frac1{\op{Li}^{-1}(n)^2}\int_2^{\op{Li}^{-1}(n)}\frac t{\log t}\,dt+O_\varepsilon\left(n^{-5/2+\varepsilon}\cdot\op{Li}^{-1}(n)^2\right).\]
    The error term here becomes $O_\varepsilon\left(n^{-1/2}\right)$ after adjusting constant factors, and we get the result after multiplying both sides by $\log n.$
\end{proof}

\begin{remark}
    One of the costs of the small error term $O_\varepsilon\left(x^{-1/2+\varepsilon}\right)$ is that the actual main term here is relatively unreadable. For example, it is no longer immediately obvious that the limit here is $1/2.$
\end{remark}
The point of the above result is that we are able to describe the behavior of the sum from \autoref{prop:goestohalf} with very high accuracy---on the order of $n^{-1/2}$---using Cram\'er's model. This is more accuracy than we could previously achieve, yet its constant factors still look surprisingly off, as in \autoref{fig:randomgoodloglog}.

\section{Generalization to Gaussian Integers} \label{sec:gaussian}

\subsection{A Motivating Prime Sum}

We now move the study of prime sums to primes in $\ZZ[i].$ One problem here is that there is no truly standard way to choose one prime while discarding its associates in the same way that we choose the positive primes in $\ZZ.$ This is a real problem because it prevents us from, for example, usefully summing
\[\sum_{\op N(\pi)\le x}\pi=0,\]
where the sum is over irreducible elements $\pi\in\ZZ[i].$ Indeed, for every prime $\pi$ with $\op N(\pi)\le x,$ we have that $-\pi$ is a distinct element which is also a prime of norm $\le x,$ so we in total sum to $0$ after constructing pairs $(\pi,-\pi).$ In fact, we still can't even usefully sum
\[\sum_{\op N(\pi)\le x}\pi^2=0\]
because, for each prime $\pi\in\ZZ[i]$ with $\op N(\pi),$ we see that $i\pi$ is another distinct prime of norm $\le x.$ In the same way as before, we can create pairs of primes $(\pi,i\pi)$ (even though $\pi\mapsto i\pi$ is not involutive, it is invertible), in total still summing to $\pi^2+(i\pi)^2=0.$

However, once we take fourth powers, then the sum becomes more interesting: the associates of a prime $\pi\in\ZZ[i]$ look like $\{\pi,i\pi,-\pi,-i\pi\},$ which all give the same fourth power, so there is no ``trivial'' cancellation. However, it turns out that there is still some (nontrivial) global cancellation, no matter what our exponent is. The main result of this section is to show the following.
\gaussianthm*
\noindent This statement essentially says that there is nontrivial cancellation among the $\pi^n/\op N(\pi)^{n/2}.$ In particular, the prime number theorem for $\ZZ[i]$ implies that
\[\sum_{\op N(\pi)\le x}1=\Theta\left(\frac x{\log x}\right),\]
so to get $o(x/\log x),$ there must be some cancellation.

We have to build some machinery to prove \autoref{thm:gaussian}. We take the following definition.
\begin{defi}[Angle of a Gaussian integer]
    Let $\alpha=a+bi\in\ZZ[i]\setminus\{0\}$ be a nonzero Gaussian integer. Then we define the angle of $\alpha$ to be the angle $\theta_\alpha\in[0,2\pi)$ such that $\alpha=\sqrt{\op N(\alpha)}e^{i\theta_\alpha}.$
\end{defi}
The main intuition for \autoref{thm:gaussian} is that we are summing
\[\frac{\pi^n}{\op N(\pi)^{n/2}}=e^{ni\theta_\pi},\]
and because we expect $\theta_\pi$ to roughly be a random angle, we expect there to be a fairly large amount of cancellation in the ``random walk.'' This intuition plus a technical result is convertible into the result.
\begin{proof}[Proof of \autoref{thm:gaussian}]
    We begin by showing that the sum is actually real. Essentially, the idea is that, if $\pi$ is a prime of bounded norm, then $\overline\pi$ also has the same bounded norm and will cancel out any imaginary contribution. To make this rigorous, we divide the sum up into the pieces
    \[\sum_{\op N(\pi)\le x}\frac{\pi^n}{\op N(\pi)^{n/2}}=\sum_{\substack{\op N(\pi)\le x\\\op{Im}\pi>0}}\frac{\pi^n}{\op N(\pi)^{n/2}}+\sum_{\substack{\op N(\pi)\le x\\\op{Im}\pi=0}}\frac{\pi^n}{\op N(\pi)^{n/2}}+\sum_{\substack{\op N(\pi)\le x\\\op{Im}\pi<0}}\frac{\pi^n}{\op N(\pi)^{n/2}}.\]
    The middle sum is entirely real, so we won't pay more attention to it. Further, notice that for every prime $\pi$ with $\op{Im}\pi>0$ and $\op N(\overline\pi)=\op N(\pi)\le x,$ we have $\overline\pi$ has $\op{Im}\overline\pi<0$ and still $\op N(\pi)\le x.$ This map lets us pair each $\pi$ in the first sum with $\overline\pi$ living in the second. Thus,
    \[\sum_{\op N(\pi)\le x}\frac{\pi^n}{\op N(\pi)^{n/2}}=\sum_{\substack{\op N(\pi)\le x\\\op{Im}\pi>0}}\left(\frac{\pi^n+\overline\pi^n}{\op N(\pi)^{n/2}}\right)+\sum_{\substack{\op N(\pi)\le x\\\op{Im}\pi=0}}\frac{\pi^n}{\op N(\pi)^{n/2}}.\]
    Now, $\pi^n+\overline\pi^n$ is self-conjugate and hence real. Thus, the entire sum is real, finishing this part of the proof.
    
    We now turn to bounding the sum. As suggested by our intuition earlier, we write $\pi=\sqrt{\op N(\pi)}e^{i\theta_\pi}$ for each $\pi,$ meaning that we want to bound
    \[S:=\sum_{\op N(\pi)\le x}\frac{\pi^n}{\op N(\pi)^{n/2}}=\sum_{\op N(\pi)\le x}e^{ni\theta_\pi}.\]
    For peace of mind, we note that we already know this sum is real, so it suffices to take real parts and estimate
    \[S=\sum_{\op N(\pi)\le x}\cos(n\theta_\pi). \label{eq:cossum} \tag{4.1}\]
    The idea, now, is to split up the various $\theta_\pi$ into small sectors and pretend that $\cos$ is constant over each sector. To make this work however, we need the technical Theorem 4 from \parencite{Kovalchik1976}.
    \begin{lem}[Koval{\textquotesingle}chik] \label{lem:technical}
        Let $x$ be a real number going to infinity, and let $S$ be a sector of angle at least $\Theta^{-1/4+\varepsilon}$ for some $\varepsilon>0.$ Then
        \[\sum_{\substack{\op N(\pi)\le x\\\theta_\pi\in S}}1=\frac\Theta{2\pi}\cdot8\op{Li}(x)\left(1+O\left(\frac1{(\log x)^B}\right)\right)\]
        for any arbitrary large constant $B.$
    \end{lem}
    We now set $N=\floor{x^{1/3}}>1$ and define the $|n|N$ sectors
    \[S_k:=\left[\frac{2\pi k}{|n|N},\frac{2\pi(k+1)}{|n|N}\right),\]
    where $k$ varies from $1$ to $|n|N.$ We note that each sector has angle exceeding $x^{-1/4+\varepsilon_0},$ so these sectors are fair play for \autoref{lem:technical}.
    
    Anyways, we take \autoref{eq:cossum} and divide up by our sectors so that we are bounding
    \[S=\sum_{k=1}^{|n|N}\sum_{\substack{\op N(\pi)\le x\\\theta_\pi\in S_k}}\cos(|n|\theta_\pi).\]
    Now, fixing one sector $S_k,$ we see $n\theta_\pi\in[2\pi k/N,2\pi(k+1)/N),$ so we have that
    \[\cos(n\theta_\pi)\in\left[\cos\frac{2\pi k}N,\cos\frac{2\pi(k+1)}N\right).\]
    This means that our sum turns into
    \[S=\sum_{k=1}^{|n|N}\sum_{\substack{\op N(\pi)\le x\\\theta_\pi\in S_k}}\left[\cos\frac{2\pi k}N+O\left(\left|\cos\frac{2\pi(k+1)}N-\cos\frac{2\pi k}N\right|\right)\right] \label{eq:split} \tag{4.2}.\]
    The error term here comes out to $-2\sin\left(\frac\pi N\right)\sin\left(\frac\pi N(2x+1)\right)$ using a sum-to-product formula; the estimate $\sin(\theta)<\theta$ for $\theta>0$ is enough to conclude that the error is $O(1/N),$ which is good enough for our purposes.
    
    Now, the terms in \autoref{eq:split} do not actually depend on the individual prime, so we may just count the number of primes by using \autoref{lem:technical} to get
    \[S=\sum_{k=1}^{nN}\frac{8\op{Li}(x)}{nN}\left(1+O\left(\frac1{(\log x)^A}\right)\right)\left(\cos\frac{2\pi k}N+O\left(\frac1N\right)\right),\]
    where $A$ is an arbitrarily large constant. Expanding this directly, we get that
    \[S=\sum_{k=1}^{nN}\left[\frac{8\op{Li}(x)}{nN}\cos\frac{2\pi k}N+O\left(\frac{8\op{Li}(x)}{nN(\log x)^A}\cos\frac{2\pi k}N\right)+O\left(\frac{8\op{Li}(x)}{nN^2}\right)+O\left(\frac{8\op{Li}(x)}{nN^2(\log x)^A}\right)\right].\]
    The first term will vanish when summed. As for the remaining error terms, we sum over $k$ and are left with
    \[S=O\left(\frac{\op{Li}(x)}{(\log x)^A}\right)+O\left(\frac{\op{Li}(x)}{N}\right)+O\left(\frac{\op{Li}(x)}{N(\log x)^A}\right).\]
    The new first error term is the main term to worry about because $N\sim x^{1/3}$ is large. This leave us with $O\left(x/(\log x)^{A+1}\right)$ because $\op{Li}(x)\sim x/\log x.$ We set $B=A+1$ to finish the proof.
\end{proof}

Looking at the above proof, we note that there is really no need to restrict $n$ to be a nonzero integer. Using the reframing with angles of Gaussian integers, the core of the above proof lies in the following result, which is proven in the same way as above.
\begin{prop} \label{prop:primwalk}
    Fix $y$ a nonzero real number. Then, for any real number $x,$
    \[\sum_{\op N(\pi)\le x}e^{iy\theta_\pi}=O\left(\frac{x}{(\log x)^A}\right),\]
    where $A$ is an arbitrarily large (but fixed) constant. Here, the sum is taken over Gaussian primes $\pi\in\ZZ[i]$ with $\op N(\pi)\le x$; the sum is a real number.
\end{prop}
\begin{remark}
    It is probably possible to generalize \autoref{prop:primwalk} further, allowing $y$ to vary with $x,$ though the authors have not worked this out explicitly.
\end{remark}
\begin{remark}
    Comparing \autoref{prop:primwalk} with Weyl's equidistribution criterion, it is unsurprising that we were forced to appeal to an equidistribution result like \autoref{lem:technical} in order to achieve \autoref{prop:primwalk}.
\end{remark}

We close this section by actually computing the case of sums of fourth powers, which follows quickly from \autoref{thm:gaussian}.
\begin{cor} \label{prop:gaussianfourth}
    For real numbers $x,$ we have that
    \[\sum_{\op N(\pi)\le x}\pi^4=O\left(\frac{x^3}{(\log x)^A}\right),\]
    where $A$ is an arbitrarily large constant.
\end{cor}
\begin{proof}
    By \autoref{thm:gaussian}, we have that
    \[\sum_{\op N(\pi)\le x}\frac{\pi^4}{\op N(\pi)^2}=O\left(\frac{x}{(\log x)^A}\right).\]
    In order to use partial summation, we rewrite the desired sum as
    \[\sum_{2\le k\le x}\left(\sum_{\op N(\pi)=k}\pi^4\right).\]
    Partial summation turns this into
    \[\sum_{2\le k\le x}\left(\sum_{\op N(\pi)=k}\frac{\pi^4}{\op N(\pi)^2}\cdot k^2\right)=x^2O\left(\frac{x}{(\log x)^A}\right)+O\left(\int_2^x\frac{t\cdot2t}{(\log t)^A}\,dt\right),\]
    for any arbitrarily large constant $A.$ Both terms are $O\left(x^3/(\log x)^A\right),$ which finishes. (The second term is safe because the maximum value of the integral is $x^2/\log x,$ for $x$ sufficiently large.)
\end{proof}

\subsection{A Random Model for Gaussian Integers}
Motivated by the previous section, we now describe a random model for the Gaussian primes, akin to Cram\'er's model. We define our probability space to be over all subsets $\Pi$ of $\CC$ satisfying the following properties.
\begin{itemize}
    \item We will fix $\pm1\pm i$ always live in $\Pi.$
    \item The set of norms $\op N(\Pi):=\left\{|z|^2:z\in\Pi\right\}$ consists exclusively of integers, where an integer $n>1$ has probability $1/(2\log n)$ of appearing as a norm.
    \item For each norm $n\in\op N(\Pi),$ there are eight elements of $\Pi.$ An angle $\theta$ is chosen randomly from $[0,\pi/4),$ and then our eight elements are $\sqrt ne^{\pm i\theta}$ and their associates.
\end{itemize}
Note that we are choosing norms with probability $1/(2\log n)$ because all but a density-zero set of the norms in $\ZZ[i]$ are $1\pmod4$ primes, which consist of roughly half (the density) of the rational primes; in other words, we have rigged the Prime number theorem for $\ZZ[i]$ into our norm distribution. Also, it is worth noting that the symmetry of associates and conjugates is built into the model.

In practice, the model behaves as two infinite tuples of random variables: the norms, which are a subset of $\ZZ$ chosen akin to a ``sparse'' Cram\'er's model; and the angles, which are essentially randomly chosen real numbers in $[0,\pi/4).$ As an example of what we can do, we have the following case of \autoref{prop:primwalk}.
\begin{prop} \label{prop:gaussiancramerex}
    Over all subsets $\Pi\subseteq\CC$ weighted according to the model above, we have that
    \[\sum_{\substack{\op N(\pi)\le x\\\pi\in\Pi}}e^{4i\theta_\pi}=O\left(x^{1/2+\varepsilon}\right)\]
    with probability $1-x^{-2}.$
\end{prop}
\begin{proof}
    We begin by rewriting without the complex numbers. For each norm $n\in\op N(\Pi),$ we let $\theta_n\in[0,\pi/4)$ be the corresponding angle. This makes the sum equal to
    \[X:=\sum_{\substack{n\in\op N(\pi)\\n\le x}}8\cos(4\theta_n)\]
    after accounting for associates and conjugates. At this point, the computation is essentially a random walk, but we will write it out for completeness; for example, there might be worries about the distribution of the norms, but this does not have an effect.
    
    As mentioned above, the $\theta_n$s are essentially random variables on their own, so we are able to immediately compute the expected value here as
    \[\mathbb E[X]=\sum_{\substack{n\in\op N(\pi)\\n\le x}}\mathbb E\left[8\cos(4\theta_n)\right]=0\]
    by linearity. We now study higher moments in order to appeal to Chebychev's inequality, as in \autoref{thm:cramerpsums}; we refer to its proof for a more explicit computation. In particular, we look at the expansion of
    \[\Bigg(\sum_{\substack{n\in\op N(\pi)\\n\le x}}\cos(4\theta_n)\Bigg)^k\]
    for an integer $k$ to be fixed later. Because the $\theta_n$ are independent variables, the expected value here will vanish for all terms in the expansion, except for those which do not contain exactly one $\cos(4\theta_n)$ term for any particular $n.$
    
    As in \autoref{thm:cramerpsums}, there are at most $O_k\left(x^{k/2}\right)$ of these terms to worry about (note that the number of norms is less than or equal to $x$), and each term is $O(1)$ at most, so we see
    \[\mathbb E\left[X^k\right]=\mathbb E\left[\Bigg|\sum_{\substack{n\in\op N(\pi)\\n\le x}}\cos(4\theta_n)\Bigg|^k\right]=O\left(x^{k/2}\right).\]
    By Chebychev's inequality, we get that
    \[\mathbb P\left[|X|\ge C_k^{1/k}x^\varepsilon\cdot x^{1/2}\right]\le x^{-\varepsilon k},\]
    for some constant $C_k.$ Setting $k$ large enough (with respect to $\varepsilon$) gets us $X=O\left(x^{1/2+\varepsilon}\right)$ with probability $1-x^{-2}.$
\end{proof}
\begin{remark}
    Technically, we have not used the distribution of norms in the above proof, but this is roughly because the summands are not very sensitive to the size or number of the primes. In contrast, if we were just counting the number of primes, then this distribution of norms will matter.
\end{remark}
Note that the above proposition continues the paradigm that Cram\'er-like models roughly allow us to reduce the error term by a $1/2$ in the exponent. This error term also fits the data; see \autoref{fig:gaussianwalk} to compare the error term for the sum over the actual Gaussian primes.
\begin{figure}[!htbp]
    \centering
    \includegraphics[width=0.7\textwidth]{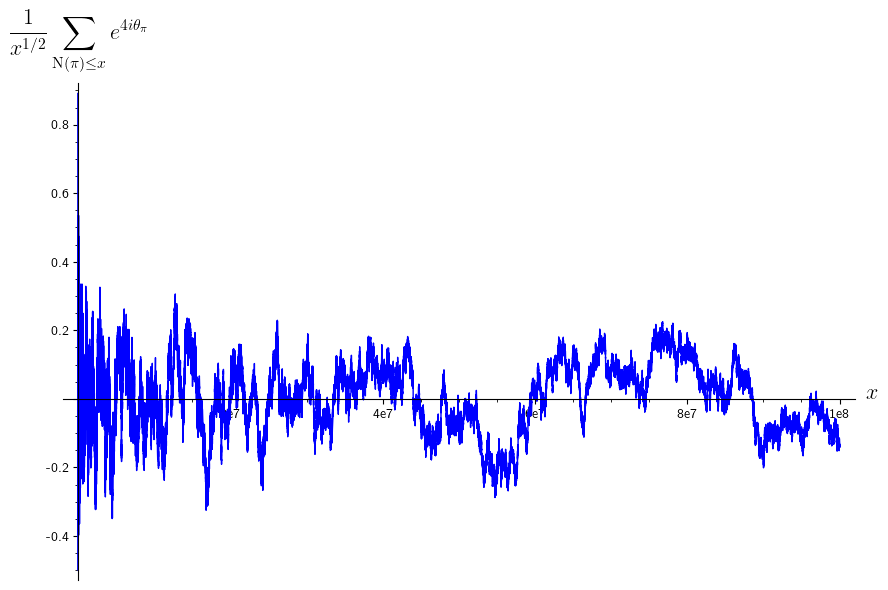}
    \caption{Sharpness of the error of \autoref{prop:gaussiancramerex}}
    \label{fig:gaussianwalk}
\end{figure}

\noindent With this in mind, we make the following conjecture.
\begin{conj}
    Fix $y$ a nonzero real number and $\varepsilon>0.$ Then, for any real number $x,$
    \[\sum_{\op N(\pi)\le x}e^{iy\theta_\pi}=O_\varepsilon\left(x^{1/2+\varepsilon}\right).\]
\end{conj}

\section{Acknowledgements}
We would like to thank our counselor Anupam for advice and patience as we worked through the project this year. In addition, we are grateful to thank Professor Lowry-Duda for the project and advice midway through the program. We also thank Professor Fried and Arya Vadnere for organizing the research labs. Lastly, we thank the PROMYS program for the opportunity to participate in the research labs, and we thank the Clay Mathematics Institute for supporting PROMYS.

\appendix

\section{Estimates with Cram\'er's Model} \label{sec:cramer}

Before doing any estimations with Cram\'er's model, we pick up the following technical lemma.
\begin{lem} \label{lem:upgradedabel}
    Let $\alpha$ be a nonnegative real number. Then
    \[\sum_{3\le n\le x}\frac{n^\alpha}{\log n}=\int_2^x\frac{t^\alpha}{\log t}\,dt+O\left(\frac{x^\alpha}{\log x}\right).\]
\end{lem}
\begin{proof}
    The idea is to integrate by parts once to make the error terms easier to control in a continuous setting, and then integrate by parts back to get the statement. Set $f(t):=t^\alpha/\log t$ for brevity; it is continuously differentiable for $t>1.$ Using summation by parts, we see that
    \[\sum_{3\le n\le x}f(n)=\floor{x}f(x)-2f(2)-\int_2^x\floor tf'(t)\,dt.\]
    Note that $\floor{x}f(x)=xf(x)+O(f(x)).$ As for the integral, we can express it as
    \[\int_2^x\floor tf'(t)\,dt=\int_2^xtf'(t)\,dt-\int_2^x\{t\}f'(t)\,dt,\]
    and the right-hand integral is bounded by $\int_2^xf'(t)\,dt=O(f(x)).$ Putting everything together, we see
    \[\sum_{3\le n\le x}f(n)=xf(x)+O(f(x))-2f(2)-\int_2^xtf'(t)\,dt+O(f(x)).\]
    The error terms combine to $O(f(x)).$ The rest can be collapsed using integration by parts into
    \[\sum_{3\le n\le x}f(n)=\int_2^xf(t)\,dt+O(f(x)).\]
    This is exactly what we wanted.
\end{proof}

The main point of this section is to prove the following statement.
\begin{thm} \label{thm:cramerpsums}
    Let $\alpha$ be a nonnegative real number. Then over all sequences $\mathcal P\subseteq\ZZ$ with probability weighted by Cram\'er's model, we have that
    \[\sum_{\substack{p\le x\\p\in\mathcal P}}p^\alpha=\int_2^x\frac{t^\alpha}{\log t}\,dt+O_\varepsilon\left(x^{\alpha+1/2+\varepsilon}\right),\]
    with probability at least $1-x^{-2}.$
\end{thm}
\begin{proof}
    This argument is not original. To use Cram\'er's model, we will have to look directly at individual integers, so we rewrite the sum as
    \[\sum_{\substack{p\le x\\p\in\mathcal P}}p^\alpha=\sum_{n\le x}n^\alpha1_{\mathcal P}(n),\]
    where $1_{\mathcal P}$ is the $\mathcal P$-indicator. This will be easier to analyze with an expected value of $0,$ so we remove the main term from the sum by writing
    \[\sum_{n\le x}n^\alpha1_{\mathcal P}(n)=\sum_{3\le n\le x}\frac{n^\alpha}{\log n}-\sum_{3\le n\le x}n^{\alpha}\left(1_{\mathcal P}(n)-\frac1{\log n}\right).\label{eq:cramsimp}\tag{A.1}\]
    The main term, which is the first sum on the right-hand side, is estimated using \autoref{lem:upgradedabel} as
    \[\sum_{3\le n\le x}\frac{n^\alpha}{\log n}=\int_2^x\frac{t^\alpha}{\log t}\,dt+O\left(\frac{x^\alpha}{\log x}\right).\]
    This error term is $O\left(x^\alpha\right),$ so it is safe. It remains to bound the error of \autoref{eq:cramsimp}, for which we set the random variables $X_n:=n^\alpha\left(1_{\mathcal P}(n)-\frac1{\log n}\right)$ for $n\ge3,$ and $X:=\sum_{3\le n\le x}X_n.$ We see that it suffices for $X=O_\varepsilon\left(x^{\alpha+1/2+\varepsilon}\right)$ with probability $1-x^{-2}.$
    
    We would like to bound $X.$ On one hand, we note $\mathbb E[X_n]=0,$ so linearity of expectation gives $\mathbb E[X]=0.$ However, we would like something more sophisticated, so we will use Chebychev's inequality, for which we need to study
    \[\mathbb E\left[|X|^k\right]=\mathbb E\left[\bigg|\sum_{3\le n\le x}X_n\bigg|^k\right],\]
    for some positive integer $k$ to be fixed later. After expanding out this sum and using linearity of expectation, the independence of the $X_n$ will imply that many terms have expected value $0.$ In fact, for a term
    \[X_1^{a_1}X_2^{a_2}\ldots X_{\floor x}^{a_{\floor x}}\text{ such that }a_1+a_2+\cdots+a_{\floor x}=k\]
    to not have expected value $0,$ each $a_n$ must either be $0$ or bigger than $1.$ In particular, each nonzero exponent is at least two, so there are at most $k/2$ exponents with nonzero entries; then there are fewer than $k^{k/2}$ ways to assign the exponents.
    
    So in total, there are fewer than $k^{k/2}\binom{\floor{x}}{k/2}\le C_kx^{k/2}$ terms to worry about, for some constant $C_k$ depending on $k.$ Further, each term consists of $k$ terms of $X_\bullet=n^\alpha\left(1_{\mathcal P}(n)-\frac1{\log n}\right)=O\left(x^\alpha\right).$ So we have, in total, is bounded by
    \[\mathbb E\left[|X|^k\right]\le C_kx^{k/2}\left(x^\alpha\right)^k,\]
    where $C_k$ is some constant depending on $k$ (but not $x$).
    
    We are now ready to use Chebychev's inequality. Setting $\varepsilon>0$ small, we see that
    \[\mathbb P\left[|X|\ge x^\varepsilon\mathbb E\left[|X|^k\right]^{1/k}\right]\le x^{-\varepsilon k}.\]
    Using our bound, we see that
    \[\mathbb P\left[|X|\ge C_k^{1/k}x^{\alpha+1/2+\varepsilon}\right]\le x^{-\varepsilon k}.\]
    With $\varepsilon>0$ fixed, we now set $k\ge2/\varepsilon$ so that $X=O_\varepsilon\left(x^{\alpha+1/2+\varepsilon}\right)$ with probability $1-x^{-2}.$ (Note the implied constant $C_k^{1/k}$ is now dependent on $\varepsilon.$) This is what we wanted.
\end{proof}
\autoref{fig:psalpha0} and \autoref{fig:psalpha1} showcase the error term from \autoref{thm:cramerpsums} on the primes. It appears that Cram\'er's model has successfully retrieved the correct magnitude of error term, though there are subtleties, such as the fact that the error largely looks negative.

\begin{figure}[!htbp]
    \centering
    \includegraphics[width=0.7\textwidth]{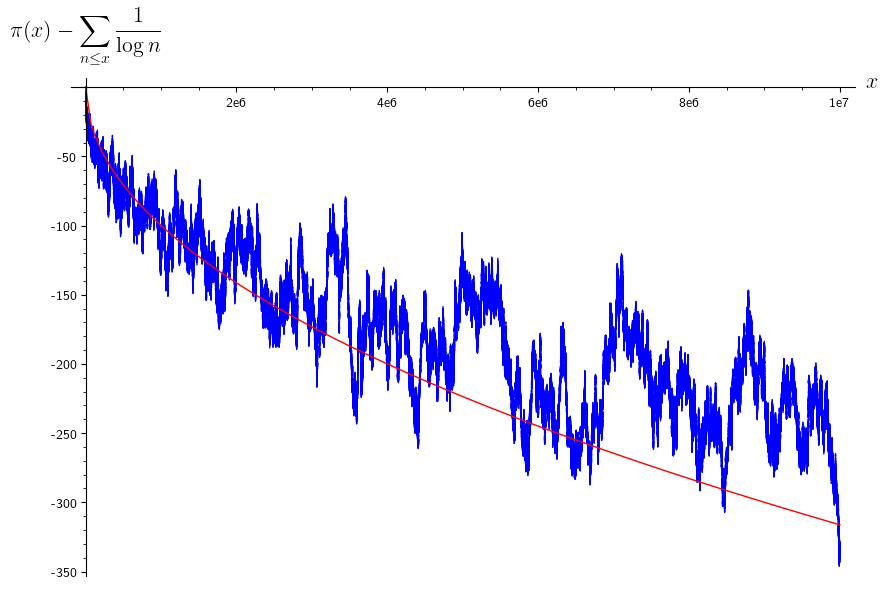}
    \caption{Verification of the error bound in \autoref{thm:cramerpsums}; $y=-x^{1/2}/10$ in red}
    \label{fig:psalpha0}
\end{figure}
\begin{figure}[!htbp]
    \centering
    \includegraphics[width=0.7\textwidth]{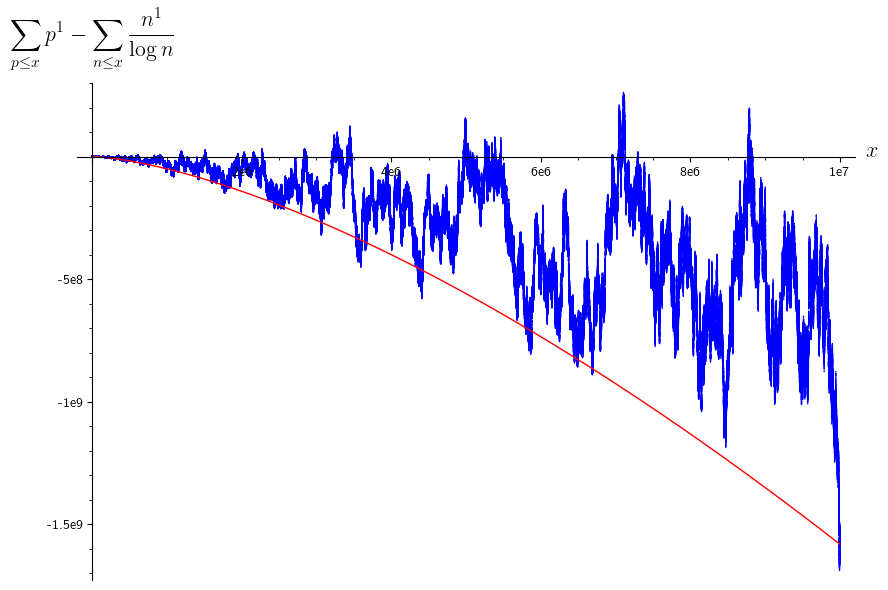}
    \caption{Verification of the error bound in \autoref{thm:cramerpsums}; $y=-x^{3/2}/20$ in red}
    \label{fig:psalpha1}
\end{figure}

\newpage
\printbibliography[title={References}]

\end{document}